\documentclass{amsart}%
\usepackage{amssymb,mathabx}
\usepackage{color}
\usepackage{amsmath}
\usepackage{graphicx}
\usepackage{amsfonts}%
\usepackage[colorlinks=true]{hyperref}
\usepackage{bm}
\usepackage{amsfonts, stmaryrd}
\usepackage[bbgreekl]{mathbbol}
\DeclareSymbolFontAlphabet{\mathbbm}{bbold}
\DeclareSymbolFontAlphabet{\mathbb}{AMSb}
\usepackage{mathrsfs}

\newcommand{\EE}{\ensuremath{\mathbb{E}}}

\newcommand{\NN}{\ensuremath{\mathbb{N}}}

\newcommand{\PP}{\ensuremath{\mathbb{P}}}
\newcommand{\QQ}{\ensuremath{\mathbb{Q}}}
\newcommand{\RR}{\ensuremath{\mathbb{R}}}

\newcommand{\XX}{\ensuremath{\mathbb{X}}}
\newcommand{\YY}{\ensuremath{\mathbb{Y}}}

\newcommand{\de}{\delta}
\renewcommand{\phi}{\varphi}

\newcommand{\Om}{\ensuremath{\Omega}}
\newcommand{\om}{\ensuremath{\omega}}

\newcommand{\w}{\omega}

\newcommand{\bB}{\ensuremath{\mathcal{B}}}
\newcommand{\cC}{\ensuremath{\mathcal{C}}}

\newcommand{\fF}{\ensuremath{\mathcal{F}}}

\newcommand{\pP}{\ensuremath{\mathcal{P}}}

\newtheorem{theorem}{Theorem}

\newtheorem{corollary}[theorem]{Corollary}

\newtheorem{definition}[theorem]{Definition}
\newtheorem{example}[theorem]{Example}

\newtheorem{lemma}[theorem]{Lemma}

\newtheorem{remark}[theorem]{Remark}

\newcommand{\ltn}{\ensuremath{\left| \! \left| \! \left|}}
\newcommand{\rtn}{\ensuremath{\right| \! \right| \! \right|}}

\title[Rough Path Theory to approximate random dynamical systems]
{Rough Path Theory to approximate random dynamical systems}

\author{H. Gao}
 \address[Hongjun Gao]{Institute of  Mathematics\\
School of Mathematical Sciences, Nanjing Normal University\\Nanjing 210046, China\\ }
\email[Hongjun Gao]{05195@njnu.edu.cn}

\author{M.J.Garrido--Atienza}
\address[Mar\'{\i}a J. Garrido--Atienza]{Facultad de Matem\'aticas\\
Avenida Reina Mercedes, s/n, 41012, Sevilla, Spain\\
 }\email[Mar\'{\i}a J. Garrido--Atienza]{mgarrido@us.es}

\author{A. Gu}
\address[Anhui Gu]{School of Mathematics and Statistics\\
Southwest University, Chongqing 400715, China}\email[Anhui Gu]{gahui@swu.edu.cn}

\author{K. Lu}
\address[Kening Lu]{346 TMCB\\
Brigham Young University, Provo, UT 84602, USA} \email[Kening
Lu]{klu@math.byu.edu}

\author{B. Schmalfu{\ss }}
\address[Bj{\"o}rn Schmalfu{\ss }]{Institut f\"{u}r Stochastik\\
Friedrich Schiller Universit{\"a}t Jena, Ernst Abbe Platz 2, D-77043\\
Jena, Germany\\
 }\email[Bj{\"o}rn Schmalfu{\ss }]{bjoern.schmalfuss@uni-jena.de}

\begin{document}
\maketitle
\begin{abstract}
We consider the rough differential equation $dY=f(Y)d\bm \om$ where $\bm \om=(\omega,\bbomega)$ is  a rough path defined by a Brownian motion $\omega$ on $\RR^m$. Under the usual regularity assumption on $f$, namely $f\in C^3_b (\RR^d, \RR^{d\times m})$, the rough differential equation has a unique solution that defines a random dynamical system $\phi_0$. On the other hand, we also consider an ordinary random differential equation $dY_\delta=f(Y_\delta)d\omega_\de$, where $\omega_\de$ is a random process  with stationary increments and continuously differentiable paths that approximates $\omega$. The latter differential equation generates a random dynamical system $\phi_\delta$ as well. We show the convergence of the random dynamical system $\phi_\delta$ to $\phi_0$ for $\delta\to 0$ in H\"older norm.
\end{abstract}

\section{Introduction}\label{s0}

The theory of random dynamical systems (RDS) allows to study the qualitative longtime behavior of differential equations containing random terms.
For a comprehensive overview of this theory we refer to Arnold \cite{MR1723992}.
There are at least two classes of these differential equations with random terms. One class is given by  ordinary differential equations satisfying (roughly speaking) existence and uniqueness conditions with noise dependent coefficients (the meaning of noise is given in Definition
\ref{d1} below). The second class are stochastic differential equations. The solutions of these equations in general contain a stochastic integral in the sense of Ito or Stratonovich, defined as a limit in probability of special Riemann sums related to the integrand and the increments of the noise, see for a precise definition Kunita \cite{MR1070361}. Remarkably, these integrals are defined only almost surely. This leads to the fact that the solutions are defined almost surely where exceptional sets may  depend on the initial condition, causing that the solutions of these equations do not generate an RDS, see Definition \ref{d2}. To derive from an Ito (Stratonovich) equation an RDS some effort is necessary. In particular one has to combine the fact that a stochastic differential equation generates a stochastic flow, see Kunita \cite{MR1070361}, to further apply a perfection technique that gives a version of the solution random field which is an RDS, see Kager and Scheutzow \cite{MR1485117}.

\smallskip

A way to avoid exceptional sets is to consider a pathwise definition of the stochastic integral based on fractional derivatives. For H\"older continuous driven processes with H\"older exponent bigger than $1/2$ the definition of the integral is nothing else than the well-known Young integral, see Young \cite{MR1555421} and Z\"ahle \cite{MR1640795}. In this framework, there are already several investigations proving that the (pathwise) solutions driven by these processes generate  an RDS, and even stating for instance the existence of random attractors and invariant manifolds that describe the longtime behavior of the solution, see Chen {\it et al.} \cite{MR3072986}, Gao {\it et al.} \cite{MR3226746}, Garrido-Atienza {\it et al.} \cite{MR2660869}, \cite {MR2593602} and \cite{MR2738732}. For less regular processes, much less is known. Nevertheless, there are already some results establishing that the solution to evolution equations driven by H\"older continuous functions with H\"older exponent in $(1/3,1/2]$ generates an RDS, see Garrido-Atienza {\it et al.} \cite{MR3479690}. The main tool used in that paper is a compensation of fractional derivatives. Note that the case of a Brownian motion is included in these considerations. An alternative method to obtain from a stochastic differential equation an RDS is to interpret such an equation in the rough path sense. Indeed the stochastic integral related to this theory is pathwise as well, so that one can avoid exceptional sets depending  on the initial states. The cocycle property then follows easily from the additivity and a shift property of the rough path integral. In this context we would like to mention the recent papers Bailleul {\it et al.} \cite{MR3624539} and Hesse and Neamtu \cite{MR}.

\smallskip

The main goal of this article is to approximate the RDS generated by a rough equation by RDS generated by random ordinary differential equations.
In order to do that, we shall have to calculate the difference of solutions of two rough differential equations driven by different noises, for which it will be necessary to estimate the distance of the two noises in a particular metric, that in fact measures the distance of the two noise paths via the H{\"o}lder norm but also the distance
of the corresponding L\'evy areas (that are second order processes) of the lifted noises in H{\"o}lder norm. Then in a first step we have to show that these L\'evy areas are well-defined. To be more precise, for a path $\omega$ of the Brownian motion we define the  stationary Gaussian noise $\omega_\delta$ such that
$\omega^\prime_\delta(r):=\frac{1}{\delta}\theta_r\omega(\delta)$ such that $\delta>0$ and $\theta_r$ is defined by (\ref{ws}) below, and consider the Gaussian process
$X_\delta:=\omega_\delta-\omega$. We then show that this random process has a particular increment covariance so that the stochastic integration theory by Friz and Hairer \cite{MR3289027} or by Friz and Victoir \cite{MR2604669} can be applied.  As a result, we obtain that the moments of the L\'evy areas of $X_\delta$ converge to zero when $\delta$ goes to zero. This then gives us the convergence of ${\bf \omega}_\delta$ to ${\bf \omega}$ and their corresponding L\'evy areas as well. Hence the solution of the differential equation driven by $\omega_\delta$ converges to the solution of the rough equation driven by a path of the Brownian motion $\omega$ and the same is true for the RDS generated by these equations.

\smallskip

This kind of approximation by $\omega_\delta$ has  been  used for stochastic ordinary and partial differential equations when the noise is very simple, see Gu {\it et al.} \cite{MR3779632} and \cite{MR3918170}. In these papers the authors prove the convergence of the attractors of the RDS driven by $\omega_\delta$ to the attractor of the RDS of the original equation. Due to the fact that the noise in these papers is trivial (additive noise or linear multiplicative noise), the authors apply a transformation technique rewriting the original stochastic equation into a random equation. In particular, with this stra\-tegy the convergence of the L\'evy area is not necessary. With the techniques developed in the current paper, in forthcoming investigations we will try to obtain similar results regarding convergence of random attractors, but without transformations of the stochastic system into a random one and for more complicated nontrivial diffusion coefficients.

\smallskip
The article is organized as follows: In Section \ref{s1} we introduce the definitions of rough path, metric dynamical system and RDS. In Section \ref{s2} we define stochastic integrals based on the increment covariance and formulate the Kolmogorov criterion for rough paths. In Section \ref{s3} we apply these results to obtain that the moments of the L\'evy area of $X_\delta$ converge to zero to further establish that the area related to ${\bf \omega}_\delta$ converges to the area related to ${\bf \omega}$ when $\delta\to 0$. This interesting fact will be exploited to prove the convergence of the RDS generated by the solution of the equation driven by $\omega_\delta$ to the RDS generated by the solution of the equation driven by $\omega$, in the case that $\omega$ is the Brownian motion. In the \hyperref[appn]{Appendix} we calculate the covariance of $X_\delta$ and $\omega_\delta$ when $\omega$ is a fractional Brownian motion with any Hurst parameter $H\in (0,1)$.

\section{Preliminaries on Rough Path Theory and Random Dynamical Systems}\label{s1}

For a given $T_1<T_2 \in \RR$, denote by $\Delta[T_1,T_2]$ the simplex $\Delta[T_1,T_2]=\{(s,t): T_1\leq s\leq t\leq T_2\}.$
Let $X$ be a mapping from $[T_1,T_2]$ to $\RR^m$ and $\XX$ be a mapping from $\Delta[T_1,T_2]$ to $\RR^{m\times m}$. These mappings are called H{\"o}lder continuous  if the seminorms given by
$$\ltn X\rtn_\beta :=\sup_{(s,t)\in \Delta[T_1,T_2], s\neq t} \frac{|X(t)-X(s)|}{(t-s)^\beta}, \quad \ltn \XX\rtn_{2\beta} :=\sup_{(s,t)\in \Delta[T_1,T_2], s\neq t} \frac{|\XX(s,t)|}{(t-s)^{2\beta}}$$
are finite. In that case, we write $X\in C^{\beta}([T_1,T_2],\RR^m)$ and  $\XX\in C^{2\beta}( \Delta[T_1,T_2],\RR^{m \times m})$.

\begin{definition}A $\beta$--rough path on an interval $[T_1,T_2]$ with values in $\RR^m$ consists of a function $X\in C^{\beta}([T_1,T_2],\RR^m)$, $\beta\in (1/3,1/2)$, as well as a second order
process $\XX\in C^{2\beta}( \Delta[T_1,T_2],\RR^{m \times m})$, subject to the Chen's algebraic relation
\begin{equation}\label{chen}
\XX(s,t)-\XX(s,u)-\XX(u,t)=(X(u)-X(s))\otimes (X(t)-X(u))
\end{equation}
for $T_1\leq s\le u\le t\leq T_2$.
\end{definition}
$X$ is usually referred as the {\it path component} whereas $\XX$ is the {\it (L\'evy) area component}.
Note that when $X$ is smooth, we can define $\XX$ by
\begin{equation*}
  \XX(s,t)=\int_s^t  (X(r)-X(s))\otimes  X^\prime(r) dr.
\end{equation*}
However, in the rough case we need to find sufficient conditions under which $X$ can be {\it enhanced} with a well-defined area component $\XX$, see Section \ref{s2}. This property is also known as $X$ can be {\it lifted} to  a rough path ${\bf X}=(X,\XX)$ and ${\bf X}$ is the lift of the process $X$.\\

For $\beta\in (1/3,1/2)$ we denote by $\mathcal C^\beta([T_1, T_2],\RR^m)$ the space of H\"older rough paths ${\bf X}:=(X,\XX)$.

If, in addition, the symmetric part of $\XX$ is given by
$$\text{Sym}(\XX(s,t))=\frac{1}{2}(X(t)-X(s))\otimes (X(t)-X(s))$$
then we will say that ${\bf X}$ is a {\it geometric} H\"older rough path and denote the corres\-ponding space by $\mathcal C^\beta_g([T_1, T_2],\RR^m)$.\\

It is interesting to point out that the space of geometric rough paths can be defined as the closure of smooth paths, enhanced with their iterated Riemann integrals, with respect to the metric $\rho_{\beta}$ defined below, in $\mathcal C^\beta([T_1, T_2],\RR^m)$, see \cite{MR3289027}, Exercise 2.12.\\

Moreover, given two rough paths ${\bf X}, {\bf Y}\in  \mathcal C^\beta([T_1, T_2],\RR^m)$ we define the H\"older rough path metric by
\begin{equation*}
  \rho_{\beta,T_1,T_2}({\bf X},{\bf Y}):=\sup_{(s,t)\in \Delta[T_1,T_2], s\neq t}\bigg(\frac{|X(t)-X(s)-Y(t)+Y(s)|}{(t-s)^\beta}+\frac{|\XX(s,t)-\YY(s,t)|}{(t-s)^{2\beta}}\bigg)
\end{equation*}
and the $\beta$-H\"older homogeneous rough path norm of $\bf X$ by
\begin{equation*}
\ltn\bf X\rtn_{\beta}:=\ltn X\rtn_{\beta}+\sqrt{\ltn \XX\rtn_{2\beta}}.
\end{equation*}

In what follows we introduce the concept of an RDS.  We start with the definition of a general noise.

\begin{definition}\label{d1}
Let $(\Omega, \fF, \PP)$ be a probability space. A metric dynamical system on $(\Omega, \fF, \PP)$ is given by a family of measurable transformations
$$\theta: (\RR\times \Omega,\bB(\RR)\otimes \fF)\to (\Omega,\fF)$$
such that
\begin{enumerate}
\item For $t_1,t_2\in \RR$, $$\theta_{t_1+ t_2}=\theta_{t_1}\circ\theta_{t_2}=\theta_{t_2}\circ \theta_{t_1}.$$
\item The measure $\PP$ is invariant with respect to $\theta$.
\item $\PP$ is ergodic with respect to $\theta_t$, that is, for any $\theta_t$-invariant set $A\in \fF$, either $\PP(A)=1$ or $\PP(A)=0$.
\end{enumerate}
\end{definition}

\begin{example}\label{ex1}

A fractional Brownian motion $B^H$ in $\RR^m$ with Hurst parameter $H\in (0,1)$ is a centered and continuous Gaussian process with covariance
\begin{equation*}
  R_{B^H}(s,t)=\frac12 (|t|^{2H}+|s|^{2H}-|t-s|^{2H}){\rm id}\quad t,\,s\in \RR,
\end{equation*}
where ${\rm id}$ is the identical matrix in $\RR^{m \times m}$, that is, we assume  that $B^H$ has independent components.
No matter the value of the Hurst parameter, we can consider a canonical version of this random process given by
\begin{equation*}
  (C_0(\RR,\RR^m),\bB(C_0(\RR,\RR^m)),\PP_H,\theta),
\end{equation*}
where $C_0(\RR,\RR^m)$ is the topological space of continuous functions which are zero at zero equipped with the compact open topology. This topology is metrizable giving us a Polish space. $\bB(C_0(\RR,\RR^m))$ denotes the Borel $\sigma$-algebra of $C_0(\RR,\RR^m)$ and the probability $\PP_H$ is the Gaussian measure of $B^H$ generated by the covariance $R_{B^H}$. For $\theta$ we consider the so-called Wiener shift defined by
\begin{equation}\label{ws}
  \theta_t\omega(\cdot)=\omega(\cdot+t)-\omega(t),\quad\text{for }t\in\RR.
\end{equation}
Then $(C_0(\RR,\RR^m),\bB(C_0(\RR,\RR^m)),\PP_H,\theta)$ is a metric dynamical system, see \cite{MR2836654}. In addition there exists a $\theta$-invariant set $\Omega$ of full $\PP_H$--measure such that $\omega\in \Omega$ is $\beta$--H{\"o}lder continuous for every $\beta<H$ on any interval $[-T,T]$. On $\Omega$ we choose $\fF$ to be the trace-algebra of $\bB(C_0(\RR,\RR^m))$ with respect to $\Omega$, and for the restriction of $\PP_H$ to this new $\sigma$-algebra we use again the same symbol $\PP_H$. In the following we will work on $(\Omega,\fF,\PP_H,\theta)$, that is also a metric dynamical system, see \cite{MR2660867}.

Finally, note that when $H=1/2$ the fractional Brownian motion reduces to the Brownian motion $B^{1/2}$.
\end{example}

Now we introduce the definition of a random dynamical system (RDS).

\begin{definition}\label{d2}
Let $(\Omega, \fF, \PP,\theta)$ be a metric dynamical system. An RDS over $(\Omega, \fF, \PP,\theta)$ is given by a measurable mapping $\varphi:\RR^+ \times \Omega \times \RR^m \mapsto \RR^m$ such that
\begin{enumerate}
\item $\varphi(0,\omega,\cdot)={\rm id},$ for all $\omega \in \Omega$.
\item The cocycle property holds true, that is,
$$\varphi(t_1+t_2,\omega,\cdot)=\varphi(t_2,\theta_{t_1} \omega,\varphi(t_1,\omega, \cdot))$$
for all $\omega \in \Omega$ and $t_1,t_2\in \RR^+$.
\end{enumerate}
\end{definition}

Later we will deal with rough differential equations, that is, differential equations that are driven by rough (non-differentiable) paths. We will use the rough path theory to give a meaning to these equations. The stochastic integral will be given in terms of compensated Riemann sums, thus the pathwise solution will enjoy the property that it generates an RDS, provided that the considered noise forms a metric dynamical system. This will be the case since in this paper we restrict our considerations to Brownian motion, whose canonical interpretation forms a metric dynamical system as we have already stated in Example \ref{ex1}. The theory of RDS will be further applied in forthcoming works to study longtime behavior properties of solutions of rough differential equations.

We would like to refer to the papers \cite{MR3624539} and \cite{MR}, where it is also investigated the generation of RDS by the solutions of rough equations.

 \section{General results on the lift of Gaussian processes}\label{s2}

In this section, we would like to remind several results (that the reader can find in the monograph \cite{MR3289027}) which shall help in order to construct the lift of a Gaussian process.

\smallskip

Let $X$ be a one dimensional process on $\RR$ and suppose that $X(t)\in L_2$. Then
\begin{equation*}
  R_X(s,t):=\EE\,X(s) X(t)
\end{equation*}
is the covariance  of $X$. We define  for $s<t$, $s'<t'$ the incremental covariance by
\begin{equation*}
  R\left(
           \begin{array}{cc}
             s & t \\
             s^\prime & t^\prime \\
           \end{array}\right)
         :=\EE(X(t)-X(s))(X(t^\prime)-X(s^\prime))
\end{equation*}
and the $\rho$-variation over an rectangle $[s,t]^2$ by
\begin{equation*}
  \|R_X\|_{\rho,[s,t]^2}:=\bigg(\sup_{\begin{array}{l}\pP(s,t),\\
  \pP^\prime(s,t)\end{array}}\sum_{\begin{array}{l}[\sigma,\tau]\in \pP,\\
  {[}\sigma^\prime,\tau^\prime{]}\in \pP^\prime\end{array}}\bigg| R\left(
           \begin{array}{cc}
             \sigma & \tau \\
             \sigma^\prime & \tau^\prime \\
           \end{array}\right)\bigg|^\rho\bigg)^{1/\rho}
\end{equation*}
for $1\le\rho<2$.

In this article, $\pP(s,t)$ denotes as usual a partition of the considered interval $[s,t]$ and $|\pP(s,t)|$ denotes the corresponding maximal mesh length.

\smallskip

The finiteness of the  $\rho$-variation over rectangles allows us to define a stochastic integral for Gaussian processes, see Friz and Hairer \cite{MR3289027}, Proposition 10.3.

\begin{theorem}\label{t1}
Let $X,\,Y$ be two continuous centered independent Gaussian processes such that for some $1\le\rho<2$
\begin{equation*}
  \|R_X\|_{\rho,[s,t]^2},\,\|R_Y\|_{\rho,[s,t]^2}<\infty.
\end{equation*}
Then for every $0\leq s\le t$  the stochastic integral
\begin{equation*}
  \int_s^t(Y(r)-Y(s))\otimes dX(r)
\end{equation*}
is defined as an $L_2$-limit of the Stieltjes integrals $\int_\pP (Y(r)-Y(s))\otimes dX(r)$ when $|\pP(s,t)| \to 0$. Moreover,
\begin{equation*}
  \EE\bigg(\int_s^t (Y(r)-Y(s))\otimes dX(r)\bigg)^2\le C \|R_X\|_{\rho,[s,t]^2}\,\|R_Y\|_{\rho,[s,t]^2}
\end{equation*}
holds, where the constant $C$ depends on $\rho$.
\end{theorem}

\begin{theorem}\label{Th10.4}
Let $X=(X^1,\cdots,X^m)$ be a continuous centered Gaussian process in $\RR^m$ with independent components, for which there exist $\rho\in [1,\frac{3}{2})$ and $M>0$ such that for every $i\in \{1,\cdots,m\}$ and all $0\leq s\leq t\leq T$
$$\|R_{X^i}\|_{\rho,[s,t]^2}\leq M |t-s|^{1/\rho}.$$
Define, for $1\leq i< j\leq m$ and $0\leq s< t\leq T$, in the $L^2$-sense, the process
\begin{equation*}
  \XX^{i,j}(s,t)=\int_s^t (X^i(r)-X^i(s))dX^j(r)
\end{equation*}
and
\begin{align*}
  \XX^{i,i}(s,t)=&\frac12(X^i(t)-X^i(s))^2,\qquad\\
   \XX^{j,i}(s,t)=&-\XX^{i,j}(s,t)+(X^i(t)-X^i(s))(X^j(t)-X^j(s)).
\end{align*}
Then there exists a version of $X$ and $\XX$  so that for any $\beta\in (\frac{1}{3},\frac{1}{2\rho})$ we have a rough path ${\bf X}=(X,\XX)\in \mathcal C^\beta_g([0,T],\RR^m)$  almost surely.
\end{theorem}

For the proof of this result see Theorem 10.4 in \cite{MR3289027}.

\smallskip

The following result, that is an adaptation of Theorem 3.3 in \cite{MR3289027}, will allow us to measure the distance between two lifts. This theorem is a version of the well-known Kolmogorov test.

\begin{theorem}\label{t4}
Let $q\ge 1$, $1/\rho-1/q>2/3$
and ${\bf X}=(X,\XX)$ and ${\bf Y}=(Y,\YY)$ be two rough paths in $\RR^m$ such that for all $-T\leq s< t\leq T$
\begin{align*}
  &|X(t)-X(s)|_{L_{2q}}\le C|t-s|^{1/(2\rho)}, \qquad |\XX(s,t)|_{L_q}\le  C|t-s|^{1/\rho},\\
  &|Y(t)-Y(s)|_{L_{2q}}\le C|t-s|^{1/(2\rho)}, \qquad |\YY(s,t)|_{L_q}\le C|t-s|^{1/\rho},
\end{align*}
for some constant $C>0$. Assume that for some $\varepsilon>0$ and all  $-T\leq s< t\leq T$
\begin{align*}
  &|X(t)-X(s)-(Y(t)-Y(s))|_{L_{2q}}\le \varepsilon C|t-s|^{1/(2\rho)},\\
  &|\XX(s,t)-\YY(s,t)|_{L_q}\le \varepsilon C|t-s|^{1/\rho}.
\end{align*}
Let $\beta\in (\frac13,\frac{1}{2\rho}-\frac{1}{2q})$. Then $\ltn{\bf X}\rtn_\beta$, $\ltn \bf Y\rtn_\beta \in L_{2q}$
and
\begin{equation*}
  |\rho_{\beta,0,T}({\bf X},{\bf Y})|_{L_{2q}}\le \varepsilon M
\end{equation*}
where $M>0$ is independent on $\varepsilon$ and depends increasingly on $C$.
\end{theorem}

We would like to stress that Theorem 3.3 works for rough paths $(X,\XX)$ defined on $[-T,T]\times \Delta[-T,T]$ and not only for rough paths defined on  $[0,T]\times \Delta[-0,T]$ as proved in \cite{MR3289027}.\\

Further we will consider centered, continuous Gaussian processes with independent
components $X=(X^1,\cdots,X^m)$ and stationary increments. As we have seen above, for the cons\-truction
of a lift associated to $X$ it suffices to estimate the  $\rho$-variation of the rectangular incremental covariance $R_{X^i}$, for all $i\in \{1,\cdots, m\}$, see Theorem \ref{Th10.4}. To this end, it is enough to focus on one component, hence, in the next result $X$ denotes a one-dimensional path, whose law is determined by
\begin{align}\label{sig}
\sigma^2(\tau)=\EE (X(t+\tau)-X(t))^2
\end{align}
for $t\in\RR,\,\tau\ge 0$.

\begin{theorem} \label{Th10.9}
Let $X$ be a real-valued continuous centered Gaussian process with stationary increments. Assume that $\sigma^2(\cdot)$ given by (\ref{sig}) is concave and non--decreasing on $[0, h]$, for some $h > 0$. Assume also
\begin{equation}\label{si}
\sigma^2(t)\leq L t^{1/\rho}
\end{equation}
for $L>0,\rho \geq 1$ and $t\in [0,h]$. Then the incremental covariance  $R_X$ of $X$ satisfies
$$\|R_X\|_{\rho,[s,t]^2}\leq L M (t-s)^{1/\rho}$$
for all intervals $ [s, t]$ with length smaller than $h$ and  $M = M(\rho,h) > 0$.\\
If now $X=(X^1,\cdots,X^m)$ is a centered continuous Gaussian process
with independent components such that each $X^i$ satisfies the above assumptions with common values of $h, L$ and $\rho\in [1,\frac{3}{2})$, then $X$ can be lifted to ${\bf X} \in \mathcal C^\beta_g([0,T],\RR^m)$.
\end{theorem}

For the proof of this result we  refer to Theorem 10.9 and Corollary 10.10 in \cite{MR3289027}. Note that revisiting the proof of Theorem 10.9 in \cite{MR3289027},  on account of (\ref{si}) it turns out that for an interval $[s,t]$ of length smaller than $h$, for $\Pi=\{t_i\}$ and $\Pi'=\{t_j'\}$ two partitions of $[s,t]$, if we fixed $i$ then
\begin{align*}
\sum_{t'_j\in \Pi'}& \Big|\EE(X(t_{i+1})-X(t_i))(X({t'_{j+1}})-X(t'_{j}))\Big|^\rho\\
\leq & \frac{2}{3^{1-\rho}} \big(2\sigma^2(t_{i+1}-t_i)\big)^\rho +\frac{L^\rho}{3^{1-\rho}}|t_{i+1}-t_i|
\leq   \Big(\frac{2^{1+\rho}+1}{3^{1-\rho}}  \Big) L^\rho |t_{i+1}-t_i|,
\end{align*}
hence summing over $t_i$ and taking the supremum we obtain
$$\|R_X\|_{\rho,[s,t]^2}\leq L \Big(\frac{2^{1+\rho}+1}{3^{1-\rho}}  \Big)^{1/\rho} (t-s)^{1/\rho}=:LM (t-s)^{1/\rho}$$
that is, the value of the constant $M$ above is given by
\begin{equation}\label{m}
M:= \Big(\frac{2^{1+\rho}+1}{3^{1-\rho}}  \Big)^{1/\rho}.
\end{equation}

\section{Approximation of the Brownian motion by a stationary process}\label{s3}

Let us consider the metric dynamical system $(\Omega,\fF,\PP_{\frac{1}{2}},\theta)$ introduced in Example \ref{ex1} in Section \ref{s1} and denote by $\omega:=B^{1/2}(\omega) \in \Omega$ the canonical Brownian motion on $\RR$ with values in $\RR^m$.

First of all, in this section we want to enhance $\omega$ to obtain the lift $\bm \om:=(\omega,\bbomega)$. Note that any of the components $\omega^i$ of $\omega$ is such that
$$\sigma^2_{\omega^i}(u)=u,$$
for $u\ge 0$, that obviously is non--decreasing and concave on any time interval $[0,T]$. Since $\rho\in (1,2]$ we can estimate
$$\sigma^2_{\omega^i}(u)=u\leq T^{1-1/\rho} u^{1/\rho},$$
hence, applying Theorem \ref{Th10.9} we obtain
\begin{align}\label{ro}
\|R_{\omega^i}\|_{\rho,[s,t]^2} \leq  T^{1-1/\rho} M |t-s|^{1/\rho}
\end{align}
with $M$ given by (\ref{m}). For $\bbomega^{i,j}(s,t)$ we choose the interpretation of the integral in \cite{MR3289027}, Chapter 10, namely
\begin{align}
\label{omij}
\bbomega^{i,j}(s,t):=\int_s^t (\omega^i(r)-\omega^i(s)) d\omega^j(r), \quad s\leq t\in\RR^+\quad\text{a.s.}
\end{align}

Despite the fact that the results of Section \ref{s2} are formulated in $\RR^+$, for our purposes we need all the lifts to be defined on $\RR$. Therefore, we extend the rough path $\bm \om$ to $[-T,T]$ so that Chen's equality holds, that is, for all $s<0<t$
\begin{align}\label{ne}
\bbomega^{i,j}(s,t):=\bbomega^{i,j}(s,0)+\bbomega^{i,j}(0,t)-\omega^j(t)\omega^i(s),
\end{align}
which means that we can restrict to check whether $\bbomega^{i,j}(s,0)$ is well-defined for $s<0$. Taking into account that
\begin{align}
\label{omijb}
\int_{\pP(s,0)} (\omega(r)-\omega(s)) \otimes d\omega(r)=\int_{\pP(0,-s)} \theta_s\omega(r) \otimes d\theta_s\omega(r)
\end{align}
and that $\theta_s\omega$ is a Brownian motion with the same incremental covariance as $\omega$, then the limit when $|\pP(s,0)|\to 0$ of (\ref{omijb}) exists in the $L_2$ sense, which in turn implies that $\bbomega(s,t)$ exists for eve\-ry $s<t\in \RR$ due to Theorem \ref{t1}.

Furthermore, by Theorem \ref{Th10.4} we can find a rough path $\bm\om:=(\omega,\bbomega)\in \mathcal C_g^\beta([-T, T],\RR^m)$ for any $T>0$.
Defining on a set of measure zero $\bbomega(s,t)=0$ we find a version such that $\bm\om\in \cC^{\beta}_g([-T,T],\RR^m)$ for all $\omega\in \Omega$.

Moreover, in virtue of Theorem \ref{Th10.4} we can also define a version
\begin{equation}\label{to2}
\theta_\tau\bm \om:=(\theta_\tau \omega,\theta_\tau \bbomega)\in\mathcal C_g^\beta([-T, T],\RR^m)\quad\text{for all }T>0,\, \text{ a.s.}
  \end{equation}
We emphasize that $\theta_\tau \bbomega$ represents the area of the path $\theta_\tau \omega$.\\

Now for $\delta \in (0,1]$ we define the approximations
\begin{equation}\label{eq5}
  \omega_\delta(t):=\frac{1}{\delta}\int_0^t\theta_r\omega(\delta)dr=\int_0^t\frac{\omega(\delta+r)-\omega(r)}{\delta}dr, \quad t\in\RR.
\end{equation}
We can interpret
\begin{equation*}
  (t,\omega)\mapsto \omega_\delta(t)
\end{equation*}
as a random process on $(\Omega,\fF,\PP_{\frac{1}{2}},\theta)$. Also, it is readily seen that
 \begin{equation*}
    \RR \times \Omega\ni (t,\omega)\mapsto \omega_\delta^\prime(t)=\frac{1}{\delta}\theta_t\omega(\delta)
 \end{equation*}
is well-defined, hence $\omega_\delta$ can be enhanced by defining the Riemann integral
\begin{equation}\label{eq3}
  \bbomega_\de(s,t):=\int_s^t (\omega_\delta(r)-\omega_\delta(s))\otimes \omega'_\delta(r)dr,\quad s\leq t\in\RR.
\end{equation}
Then
\begin{equation*}
  ((s,t),\omega)\to \bbomega_\de(s,t)
\end{equation*}
for $s\le t \in \RR$ is a random field on $(\Omega,\fF,\PP_{\frac{1}{2}},\theta)$. Moreover, it is straightforward to prove the corresponding Chen's relation (\ref{chen}), henceforth the lift $\bm \w_\de:=(\omega_\delta,  \bbomega_\de)$ is a rough path.

We also define
$$\theta_\tau\bbomega_\delta(s,t):=\int_s^t (\theta_\tau\omega_\delta(r)-\theta_\tau\omega_\delta(s))\otimes (\theta_\tau\omega_\delta)'(r)dr,\; \tau\in \RR, \; s\leq t\in\RR,\omega\in\Omega.$$
Then it is easy to see
\begin{align}\label{to1}
\theta_\tau \bbomega_{\delta} (s,t)=\bbomega_{\delta} (s+\tau,t+\tau).
\end{align}
Now we define
\begin{equation}\label{to}
  \theta_\tau \bm \om_{\delta}:=(\theta_\tau \om_\de,\theta_\tau \bbomega_\de),\; \tau\in \RR.
\end{equation}
It is an immediate exercise to check that $\theta_\tau \bm \om_{\delta}$ is a rough path on $[-T,T]\times \Delta[-T,T]$, for any $T>0$.\\

Also for $\delta \in (0,1]$ we define the random process
$$\RR \times \Omega\ni (t,\omega)\mapsto X_\delta(t,\omega):=\omega_\delta(t)-\omega(t)\in \RR^m.$$

\begin{lemma}\label{prop}
$X_\de$ and $\omega_\delta$ are centered Gaussian processes on $\RR$ with stationary increments.
\end{lemma}

\begin{proof}
Note that the integrand defining $\omega_\delta$ is continuous on any $[-T,T]$ for any fixed $\delta$ and any $T>0$. Also the second moments of the integrand
are uniformly bounded in $[-T,T]$. Then $\omega_\delta$ (and therefore $X_\delta$ as well) is a Gaussian process, see \cite{MR1619188}, Pages 91 and 297.

Next, we prove that $X_\de$ has stationary increments. A simply computation shows that, for $s,\,t\in \RR$,
\begin{align*}
X_\de(t+s, \om)-X_\de(t, \om)
=&\frac 1\de\int_t^{t+s}\theta_r\omega(\de)dr-(\omega(t+s)-\omega(t))\\
=&\frac 1\de\int_0^{s} \theta_r\theta_t\omega(\de) dr-\theta_t\omega(s)\\
=&X_\de(s,\theta_t\omega).
\end{align*}
This means that the increment $X_\de(t+s,\omega)-X_\de(t,\omega)$ has the same distribution as
$X_\de(s,\omega)$ due to the fact that $\theta_t$ is measure-preserving. The same holds for $\omega_\delta$.
\end{proof}

Our next goal is to check that $X_\delta$ can be enhanced, in such a way that we obtain a rough path $(X_\delta,\XX_\delta)$ on $[-T,T]\times \Delta[-T,T]$, for any $T>0$.
\begin{theorem}\label{lift}
The variance of any of the components $X_\delta^i$ of $X_\delta$ satisfies
\begin{equation*}
  \sigma_{X^i_\de}^2(u)\le  \delta^{1-1/\rho} u^{1/\rho},
\end{equation*}
for $u\geq 0,\,\delta\in (0,1]$, $\rho\in [1,2)$. As a result,
\begin{equation*}
  \|R_{X_\delta^i}\|_{\rho,[s,t]^2}\le  \delta^{1-1/\rho} M |t-s|^{1/\rho}
  \end{equation*}
  with $M$ defined by (\ref{m}).
\end{theorem}

\begin{proof}
We know from Corollary \ref{c2} in the \hyperref[appn]{Appendix}, that the variance $\sigma_{X^i_\de}^2$ of each component $X^i_\de$ is given by
\begin{equation}\label{equa1}
\sigma_{X^i_\de}^2(u)=\left\{
\begin{array}{lcr}
u-\frac{u^3}{3\delta^2}, &\quad& 0\le u<\delta,\\
\frac 2 3\de,&\quad& u\ge \delta,
\end{array}
\right.\end{equation}
for $1\le i\le m$. Obviously,  $\sigma_{X^i_\de}^2(u)$ is a continuous function and it is non-decreasing. Moreover,
$\sigma_{X^i_\delta}^2(\cdot)$ is concave and from (\ref{equa1}) it is easily seen that
\begin{equation*}
  \sigma_{X^i_\de}^2(u)\le Lu^{1/\rho},\; \text{ with } L:=L(\delta,\rho)=\delta^{1-1/\rho}
\end{equation*}
for $u\ge 0$. Furthermore, as a consequence of Theorem \ref{Th10.9} and the estimate (\ref{m}), we can derive the following explicit estimate for the $\rho$-variation norm of $R_{X^i_\de}$:
\begin{equation}\label{eR}
\|R_{X^i_\de}\|_{\rho,[s,t]^2}\le \delta^{1-1/\rho} \bigg(\frac{{2}^{\rho+1}+1}{3^{1-\rho}}\bigg)^{1/\rho} |t-s|^{1/\rho}.
\end{equation}
\end{proof}
\begin{corollary}\label{coro}
For $\delta \in (0,1]$ the area
\begin{equation*}
  \XX_{\delta}(s,t):=\int_s^t (X_{\delta}(r)-X_\delta(s))\otimes dX_{\delta}(r)
\end{equation*}
is well-defined for all $0\le s\le t$. In particular we have
\begin{equation*}
  |\XX_{\delta}^{i,j}(s,t)|_{L_2}^2 \le C \delta^{2(1-1/\rho)} M^2|t-s|^{2/\rho}
\end{equation*}
where $M$ has been defined in Theorem \ref{lift}, and $C$ only depends on $\rho$.
\end{corollary}
\begin{proof}
For $i\neq j$ the above statement follows directly from Theorem \ref{t1} and (\ref{eR}).

On the other hand, for $i=j$ note that, for all $0\le s\le t$,
\begin{align*}
  |\XX_{\delta}^{i,i}(s,t)|_{L_2}^2&\le \EE|X_{\delta}^i(t)-X_{\delta}^i(s)|^4\le 3(\EE|X_{\delta}^i(t)-X_{\delta}^i(s)|^2)^2\\
 & \le 3 \delta^{2(1-1/\rho)} M^2 |t-s|^{2/\rho}.
\end{align*}
\end{proof}

\begin{remark}\label{r1}
In fact, we can consider $\XX^{i,j}_\delta(s,t)$ for all $s<t\in \RR$. For $s\leq 0\leq t$ we define for $1\le i,\,j\le m$
\begin{equation}\label{ne1}
  \XX^{i,j}_\delta(s,t):=\XX^{i,j}_\delta(s,0)+\XX^{i,j}_\delta(0,t)-X_\delta^j(t)X_\delta^i(s),
\end{equation}
hence we only need to give meaning to $\XX_\delta^{i,j}(s,0)$. From Lemma \ref{prop} we know that $X_\delta$ has stationary increments, therefore the incremental
covariance of $X_\delta$ and $\theta_s X_\delta$ are the same. We then could follow exactly the same steps as we did to define $\bbomega(s,0)$ for $s<0$.

In addition, in virtue of Corollary \ref{coro} we have that, for any $1\leq i,j\leq m$ and $s\le 0\le t$,
\begin{align*}
  |\XX_\delta^{i,j}(s,t)|_{L_2}^2\leq &3 \Big(|\XX_\delta^{i,j}(s,0)|_{L_2}^2+|\XX_\delta^{i,j}(0,t)|_{L_2}^2+\big(\EE|X_\delta^i(s)|^4\EE|X_\delta^j(t)|^4\big)^{1/2}\Big)\\
  \le &C \delta^{2(1-1/\rho)} M^2 ((-s)^{2/\rho}+t^{2/\rho})\leq C \delta^{2(1-1/\rho)} M^2 (t-s)^{2/\rho} <\infty,
\end{align*}
where in the penultimate inequality we have used the fact that the mapping $(a,b) \mapsto |b-a|^\theta$ for $\theta\geq 1$ is a control function, see \cite{MR2604669}, Page 22. Note that the last term of the right hand side above can be estimate in this way because
$$\big(\EE|X_\delta^i(s)|^4\EE|X_\delta^j(t)|^4\big)^{1/2}\le C\delta^{2-2/\rho}(-s)^{1/\rho}t^{1/\rho}\le C\delta^{2-2/\rho}((-s)^{2/\rho}+t^{2/\rho}).
$$
Note that for $s<t\le 0$, by shifting the time we have $|\XX_\delta^{i,j}(s,t)|_{L_2}=|\XX_\delta^{i,j}(0,t-s)|_{L_2}$ which can be considered similarly to the case in Corollary \ref{coro}.
\end{remark}

\smallskip

\begin{theorem} Let $\delta\in (0,1]$, $\rho>1$ and $q>1$ so that $1/\rho-1/q>2/3$. Then for every $\beta\in (\frac13,\frac{1}{2\rho}-\frac{1}{2q})$, we have $\ltn {\bm \om} \rtn_\beta$, $\ltn {\bm \om_\delta} \rtn_\beta \in L_{2q}$. Furthermore, there exists a positive constant $C_{q,\rho,T}$ that may depend on $T$, $q$ and $\rho$ such that
 $$|\rho_{\beta,s,t}( \bm \om_\delta ,\bm \om)|_{L_{2q}}\le C_{q,\rho,T} \delta^{1/2(1-1/\rho)}.$$ Therefore,
\begin{align*}
\lim_{\delta \to 0}  \EE\big( \rho_{\beta,s,t}( \bm \om_\delta , \bm \om))^{2q} =0,
\end{align*}
for any $[s,t]\subset [-T,T]$.
\end{theorem}
\begin{proof}
Assume first that $0\le s\le t$. To prove this result we are going to apply the version of Kolmogorov's theorem, see Theorem \ref{t4}, hence we have to check the corresponding moment conditions for $\bm \om_\delta$, $\bm \om$ and its difference (with a uniform constant that does not depend on $\delta$).

Throughout the proof $C$ is a constant that may depend on $\rho$ and/or $T$ but not on $\delta$, and may change from line to line. We will emphasize the dependence of these values when we think it is suitable to do it.

Notice that for $i\neq j$ we can consider the following splitting
\begin{align*}
\XX^{i,j}_\de(s,t)&=\int_s^t [(\omega^i_\de-\omega^i)(r)-(\omega^i_\de-\omega^i)(s)]d (\omega^j_\de-\omega^j)(r)\\
&=\int_s^t [\omega^i_\de(r)-\omega^i_\de(s)]d \omega^j_\de(r)-\int_s^t [\omega^i(r)-\omega^i(s)]d \omega^j_\de(r)\\
&-\int_s^t [(\omega^i_\de-\omega^i)(r)-(\omega^i_\de-\omega^i)(s)]d \omega^j(r)\\
&=: \bbomega_\de^{i,j}(s,t)-I^{1,i,j}_\delta(s,t)-I^{2,i,j}_\de(s,t).
\end{align*}
The integral $I^{1,i,j}_\de$ can be rewritten as
\begin{align*}
I^{1,i,j}_\de&=\int_s^t [\omega^i(r)-\omega^i(s)]d \omega^j(r)+\int_s^t [\omega^i(r)-\omega^i(s)]d (\omega^j_\delta-\omega^j)(r)\\
&=:\bbomega^{i,j}(s,t)+I^{3,i,j}_\delta(s,t).
\end{align*}
Hence, from the previous considerations we obtain that
\begin{align}\label{equal}
\bbomega_\de^{i,j}(s,t)-\bbomega^{i,j}(s,t)=\XX^{i,j}_\de(s,t)+I^{2,i,j}_\delta(s,t)+I^{3,i,j}_\de(s,t).
\end{align}
According to Theorem \ref{t1} we can estimate the second moments of $I^{2,i,j}_\delta$ and $I^{3,i,j}_\de$  by the
respective variation seminorms of the incremental covariances. In fact,
\begin{align}\label{i2}
\begin{split}
\EE|I^{2,i,j}_\de(s,t)|^2 &=\EE \bigg( \bigg|\int_s^t [X_\delta^i(r)-X_\delta^i(s)]d \omega^j(r)\bigg|^2\bigg)\\
&\leq C \|R_{X_\de^i}\|_{\rho,[s,t]^2}\|R_{\omega^j}\|_{\rho,[s,t]^2}
\end{split}
\end{align}
and
\begin{align}\label{i3}
\begin{split}
\EE|I^{3,i,j}_\de(s,t)|^2 &=\EE \bigg( \bigg|\int_s^t [\omega^i(r)-\omega^i(s)]d X_\delta^j(r)\bigg|^2\bigg)\\
&\leq C \|R_{\omega^i}\|_{\rho,[s,t]^2} \|R_{X_\de^j}\|_{\rho,[s,t]^2}.
\end{split}
\end{align}

Note that (\ref{eR}) and (\ref{ro}) together with (\ref{i2}) and (\ref{i3}) imply that
\begin{align*}
\EE|I^{k,i,j}_\de(s,t)|^{2} \leq  C \delta^{1-1/\rho} T^{1-1/\rho} M^2 |t-s|^{2/\rho },
\end{align*}
for $k=2,3$. Going back to (\ref{equal}), taking into account Corollary \ref{coro}, we obtain that
\begin{align}\begin{split}\label{dif}
\EE| \bbomega_\de^{i,j}(s,t)-\bbomega^{i,j}(s,t)|^2 &\le  3 \EE|\XX^{i,j}_\de(s,t)|^2+3\EE|I^{2,i,j}_\delta(s,t)|^2+3\EE|I^{3,i,j}_\de(s,t)|^2 \\
&\leq C M^2 \delta^{1-1/\rho} (\delta^{1-1/\rho}+T^{1-1/\rho}) |t-s|^{2/\rho}\\
&\leq C M^2 \delta^{1-1/\rho} (1+T^{1-1/\rho}) |t-s|^{2/\rho}.
\end{split}
\end{align}

In fact (\ref{dif}) can be rewritten as
\begin{align}\begin{split}\label{dif1}
| \bbomega_\de^{i,j}(s,t)-\bbomega^{i,j}(s,t)|_{L_2} &\le  C \varepsilon |t-s|^{1/\rho},
\end{split}
\end{align}
where $\varepsilon:=\delta^{1/2(1/-1/\rho)}$ and $C$ above is a constant that depends on $T$ and $\rho$, but not on $\delta$.

Moreover, for $i\neq j$, on account of Theorem \ref{t1} and (\ref{ro}), we obtain
\begin{align*}
\EE|\bbomega^{i,j}(s,t)|^2 &\le C T^{2(1-1/\rho)} M^2 |t-s|^{2/\rho}.
\end{align*}
As a result of the previous inequality and (\ref{dif}) we also obtain the corresponding estimate for $ \bbomega_\de$, namely
\begin{align*}
\EE| \bbomega_\de^{i,j}(s,t)|^2 &\le  2\EE| \bbomega_\de^{i,j}(s,t)-\bbomega^{i,j}(s,t)|^2+2\EE| \bbomega^{i,j}(s,t)|^2\\
&\leq C M^2 \Big(\delta^{2(1-1/\rho)}+\delta^{1-1/\rho} T^{1-1/\rho}+T^{2(1-1/\rho)}\Big) |t-s|^{2/\rho},
\end{align*}
and therefore
\begin{align}\label{dif2}
|\bbomega^{i,j}(s,t)|_{L_2} \le C |t-s|^{1/\rho},\quad  |\bbomega_\de^{i,j}(s,t)|_{L_2} \le C |t-s|^{1/\rho}.
\end{align}

\bigskip

On the other hand, since $\omega^i(t)-\omega^i(s)$ and $\omega_\delta^i(t)-\omega^i_\delta(s)$ are Gaussian variables, see Lemma \ref{prop}, taking into account that $\omega_\de^i(t)=X_\de^i(t)+\omega^i(t)$ we can estimate
\begin{align*}
   \EE|\omega_\de^i(t)&-\omega_\de^i(s)+\omega^i(t)-\omega^i(s)|^4\le 8\EE|\omega_\de^i(t)-\omega_\de^i(s)|^4+8\EE|\omega^i(t)-\omega^i(s)|^4\\
   \le& 24(\EE|\omega_\de^i(t)-\omega_\de^i(s)|^2)^2+24(\EE|\omega^i(t)-\omega^i(s)|^2)^2\\
    \le& 96(\EE|X_\de^i(t)-X_\de^i(s)|^2+\EE|\omega^i(t)-\omega^i(s)|^2)^2+24(\EE|\omega^i(t)-\omega^i(s)|^2)^2\\
    \leq &C |t-s|^{2/\rho},
\end{align*}
where $C$ is uniform with respect to $\delta\in (0,1]$, and depends on $T$ and $\rho$. Note now that, by a direct computation, when $i=j$ we obtain
\begin{align*}
\EE(\bbomega_\de^{i,i}(s,t)-\bbomega^{i,i}(s,t))^2&=\frac{1}{4}\EE\bigg((\omega_\de^i(t)-\omega_\de^i(s))^2-(\omega^i(t)-\omega^i(s))^2\bigg)^2\\
&=\frac{1}{4}\EE\bigg((\omega_\de^i(t)-\omega_\de^i(s)+\omega^i(t)-\omega^i(s))(X_\de^i(t)-X_\de^i(s))\bigg)^2\\
&\le \frac14 \bigg(\EE|(\omega_\de^i(t)-\omega_\de^i(s)+\omega^i(t)-\omega^i(s)|^4\EE|(X_\de^i(t)-X_\de^i(s)|^4\bigg)^{1/2}
\end{align*}
and therefore
\begin{align}\label{dif3}
|\bbomega_\de^{i,i}(s,t)-\bbomega^{i,i}(s,t)|_{L_2} \leq C \varepsilon |t-s|^{1/\rho},
\end{align}
where we remind that  $ \varepsilon=\delta^{1/2(1/-1/\rho)}$.
According to \eqref{ro} we obtain
\begin{align}\label{dif4}
|\bbomega^{i,i}(s,t)|_{L_2} \le C |t-s|^{1/\rho}.
\end{align}
Again, by using triangle inequality
\begin{align}\label{dif5}
| \bbomega_\de^{i,i}(s,t)|_{L_2} \le C |t-s|^{1/\rho}
\end{align}
and therefore, taking into account (\ref{dif1})-(\ref{dif5}), all the moment conditions for the areas are satisfied in the particular case that $q=2$ (see Theorem \ref{t4}).

Now by the equivalence of $L_q$- and $L_2$-norms on Wiener--Ito chaos, see \cite{MR2604669}, Appendix D, for every $q>1$ we have
\begin{align*}
|\bbomega(s,t)|_{L_q} &\le C_{q,\rho,T} |t-s|^{1/\rho},\\
|\bbomega_\de(s,t)|_{L_q} &\le C_{q,\rho,T} |t-s|^{1/\rho},\\
  |\bbomega_\de(s,t)-\bbomega(s,t)|_{L_q}&\le \varepsilon C_{q,\rho,T}|t-s|^{1/\rho}.
\end{align*}
On the other hand, by Theorem \ref{lift}
\begin{equation*}
  |\omega_\delta(t)-\omega_\delta(s)-(\omega(t)-\omega(s))|_{L_{2q}}=|X_\delta(t)-X_\delta(s)|_{L_{2q}}\le \varepsilon C_{q,\rho,T}|t-s|^{1/(2\rho)},
\end{equation*}
and by (\ref{ro})
\begin{align*}
 |\omega(t)-\omega(s)|_{L_{2q}} \leq C_{q,\rho,T} |t-s|^{1/(2\rho)}.
\end{align*}
These last two inequalities trivially imply that
\begin{align*}
 |\omega_\delta(t)-\omega_\delta(s)|_{L_{2q}} \leq C_{q,\rho,T} |t-s|^{1/(2\rho)},
\end{align*}
see also Lemma \ref{l2} in the \hyperref[appn]{Appendix}.

Now for every $\beta\in (1/3,1/2)$ we can find $\rho>1$ and $q>1$ so that $1/\rho-1/q>2/3$ and Theorem \ref{t4} can be applied, ensuring that $\ltn {\bm \om} \rtn_\beta$, $\ltn {\bm \om_\delta} \rtn_\beta \in L_{2q}$ and
\begin{align}\label{ee}
|\rho_{\beta,s,t}( \bm \om_\delta ,\bm \om)|_{L_{2q}}\le C_{q,\rho,T} \delta^{1/2(1/-1/\rho)}.
\end{align}
Hence,
\begin{equation*}
\lim_{\delta \to 0}  |\rho_{\beta,s,t}( \bm \om_\delta,\bm \om)|_{L_{2q}}=0.
\end{equation*}
So far, we have considered only the case $0\leq s\leq t$, but the other cases are qualitatively the same. In fact, if $s\leq 0 \leq t$ by (\ref{ne}) and (\ref{ne1}) we should only estimate the second moments of $I^{2,i,j}_\delta$ and $I^{3,i,j}_\de$ defined above. For instance, for $I^{2,i,j}_\de$ we observe
$$I^{2,i,j}_\de(s,t)=\int_s^0 [X_\delta^i(r)-X_\delta^i(s)]d \omega^j(r)+\int_0^t X_\delta^i(r)d \omega^j(r)
-X_\delta^i(s)\omega^j(t),$$
that is,
$$I^{2,i,j}_\de(s,t)=I^{2,i,j}_\de(s,0)+I^{2,i,j}_\de(0,t)-X_\delta^i(s)\omega^j(t),$$
and the second moments of this splitting can be estimated using Theorem \ref{t1} and the fact that the increments of $X_\delta$ and $\omega$ are stationary.

The case $s< t \leq 0$ could treated taking into account that $|\XX_\delta^{i,j}(s,t)|_{L_2}=|\XX_\delta^{i,j}(0,t-s)|_{L_2}$, therefore we omit it here.
\end{proof}


\begin{theorem}\label{l3}
Consider a sequence $(\delta_i)_{i\in\NN}$ that converges sufficiently fast to zero when $i\to \infty$. Then we have the following convergences
\begin{align}\begin{split}\label{conv}
\lim_{i\to\infty}&\omega_{\delta_i}=\omega\,\text{ in } C^\beta([-T,T],\RR^m),\\
  \lim_{i\to\infty}&\bbomega_{\delta_i}=\bbomega\, \text{ in }   C^{2\beta}(\Delta[-T,T],\RR^{m\times m})
  \end{split}
\end{align}
for every $T>0$, for all $\omega \in \Omega$ and $\beta<1/2$.
\end{theorem}

\begin{proof}
The assumption that the sequence $(\delta_i)_{i\in\NN}$ converges sufficiently fast to zero has the meaning that we can apply the Borel-Cantelli lemma, therefore we will restrict here to consider $\delta_i=2^{-i}$. In addition, it is sufficient to fix an arbitrary $T=n\in\NN$.

Taking into account (\ref{ee}), for every $\beta\in (1/3,1/2)$ we can find $\rho>1$ and $q>1$ so that $1/\rho-1/q>2/3$ and Theorem \ref{t4} can be applied, hence there exists a positive constant $C_n$ such that
\begin{align}\label{t2a1}
\begin{split}
 \EE\big(\rho_{\beta,-n,n} &( \bm \om_{\delta_i} ,\bm \om)\big)^{2q} \le  C_{q,\rho,n} \delta_i ^{q(1-1/\rho)}= C_{q,\rho,n} 2^{iq(1/\rho-1)}\to 0,
 \end{split}
\end{align}
as $i\to \infty$.
Now, by the Chebyshev inequality and \eqref{t2a1}, we have  that
\begin{align*}
\mathbb P_{\frac 12}\big(\omega: \rho_{\beta,-n, n}(\bm \om_{\delta_i} ,\bm \om)>2^{i/2(1/\rho-1)}\big)&\le \frac{1}{2^{iq/2(1/\rho-1)}} \EE\big(\rho_{\beta,-n,n} ( \bm \om_{\delta_i} -\bm \om)\big)^{2q}\\
&\le  C_{q,\rho,n} 2^{iq/2(1/\rho-1)}.
\end{align*}
By Borel-Cantelli lemma, we conclude that there exist a set $\Om^{(n)}\subset  \Omega$ of full $\mathbb P_{\frac 12}$--measure and $i_0=i_0(\omega,n)\ge 1$ such that for every $\omega\in \Om^{(n)}$,
\begin{equation*}
\rho_{\beta,-n, n}( \bm \om_{\delta_i} ,\bm \om)\le 2^{i/2(1/\rho-1)}
\end{equation*}
for $i\ge i_0$.
Now, taking $\hat\Om{^0}=\bigcap\limits_{n\ge 1} \Om^{(n)}$, we have $\mathbb P_{\frac 12}(\hat\Om{^0})=1$.
Replacing $\omega$ by $\theta_\tau\omega$ we can introduce the full set $\hat\Omega^\tau,\,\tau\in\RR$.\\

In order to find an invariant set $\Om'$ where the convergences (\ref{conv}) take place, let $\Omega^\tau$ be the measurable set with respect to the sigma-algebra $\fF$ of all $\omega$ such that
$(\theta_\tau \bm\om_{\delta_i})_{i\in\NN}$ converges in
$\cC^\beta([-n,n],\RR^m)$ for any $n\in\NN$, i.e. the sequence forms a Cauchy sequence for any $n$.  The measurability of $\Omega^\tau$ follows because
\begin{align*}
  &\sup_{-n\le s<t\le n}\frac{|\theta_\tau\bbomega_{\delta_i}(s,t)|}{|t-s|^{2\beta}}=\sup_{\{-n\le s<t\le n\}\cap \QQ}\frac{|\theta_\tau\bbomega_{\delta_i}(s,t)|}{|t-s|^{2\beta}},\,\text{and}\\
  &C_0(\RR,\RR^{m})\ni\omega\mapsto \theta_\tau\bbomega_{\delta_i}(s,t)
\end{align*}
is measurable. For the path component we can argue in a similar way, hence we only consider the area component. In addition the set of $\omega$ for which a sequence of random variables is a Cauchy sequence is measurable. The sets $\Omega^\tau$ contain $\hat\Omega^\tau$, hence $\Omega^\tau$ are sets of full measure.
Furthermore, if $\omega\in \Omega^\tau$, the limit is in
$\cC^\beta([-n,n],\RR^m)$ for $n\in\NN$ and an indistinguible version of
$\theta_\tau\bm\om$ is given by \eqref{to2}, which follows by the Borel-Cantelli argument as above.
On the zero-measure set ${(\Omega^\tau)}^c$ we set $\omega\equiv 0$ such that the corresponding $\bm\om$ is the zero path. We denote this new version by  the same symbols:
$\theta_\tau\bm\om=(\theta_\tau\omega,\theta_\tau\bbomega)$.
In virtue of (\ref{to1}), for any $q,\,\tau \in \RR$,
$$\theta_{\tau+q}\bbomega_{\delta_i}(s,t)=\theta_\tau\bbomega_{\delta_i}(s+q,t+q),$$
and thus, taking limits we obtain
$$\theta_{\tau+q}\bbomega(s,t)=\theta_\tau\bbomega(s+q,t+q)$$
for all $s<t\in \RR$, such that $\Omega^\tau \subset \Omega^{\tau+q}$. In a similar way,
\begin{equation*}
  \theta_{\tau}\bbomega_{\delta_i}(s,t)=\theta_{\tau+q-q}\bbomega_{\delta_i}(s,t)=\theta_{\tau+q}\bbomega_{\delta_i}(s-q,t-q)
\end{equation*}
which implies that $\Omega^\tau \supset \Omega^{\tau+q}$, and therefore taking $\tau=0$
$$\Omega^q=\Omega^0.$$ Then  for $q\in\RR$
\begin{eqnarray*}
  \theta_{-q}\Omega^0 &=& \theta_q^{-1}\Omega^0 \\
   &=& \{\omega\in\Omega:\theta_{q}\bbomega_{\delta_i}\;\text{converges for }n\in\NN\,\text{with metric }\rho_{\beta,-n,n}\}=\Omega^q=\Omega^0,
\end{eqnarray*}
hence $\Omega^\prime:=\Omega^0$ is $\theta$ invariant. We have defined $\omega\equiv 0$ and $\bm\om$ to be the zero path, hence the convergences of the statement holds true for all $\omega \in \Omega$.

\end{proof}

\section{Rough path stability of the random dynamical systems}

In this section we present the main result of this paper, concerning the convergence when $\delta \to 0$ of the solution process of a differential equation driven by $\bm \om_\delta$ to the solution process of the rough equation driven by $\bm \om$.  We will present this convergence result under the cocycle formulation.\\

As we already mentioned, we use the rough path theory to define the integral with integrator $\bm \om$. We recall therefore some of the basic facts regarding this theory, namely the definition of a controlled rough path and the sewing lemma.

\begin{definition} For $\beta \in (1/3,1/2)$, given $\omega \in C^\beta([0, T], \RR^m)$ we say that $Y\in C^\beta([0, T], \RR^d)$
is controlled by $\w$ if there exist $Y'\in C^\beta([0, T], \RR^{d\times m})$ and $R^Y \in C^{2\beta}(\Delta[0, T], \RR^{d})$ such that
\begin{equation*}
Y(t)-Y(s)=Y'(s)(\omega(t)-\omega(s))+R^Y(s,t),
\end{equation*}
for $0\leq s\leq t\leq T$. $Y'$ is known as the Gubinelli derivative of $Y$ and $R^Y$ is the remainder term.
\end{definition}
We introduce the space of controlled rough paths
$$(Y, Y')\in  \mathscr D_\om^{2\beta}([0, T], \RR^d),$$
with the norm
$$\|Y, Y'\|_{\w, \beta}:=\ltn Y'\rtn_\beta+\ltn R^Y\rtn_{2\beta}+|Y(0)|+|Y^\prime(0)|.$$
Then $ (\mathscr D_\om^{2\beta}([0, T], \RR^{d}),\|\cdot\|_{\w,\beta})$ is a Banach space, see \cite{MR3289027}, Page 56.\\

Taking into account the definition of the Wiener shift (\ref{ws}), it is straightforward to see that that  $(Y,Y')\in \mathscr D_\om^{2\beta}([\tau, \tau+T], \RR^d)$ if and only if $(Y(\cdot+\tau),Y'(\cdot+\tau))\in \mathscr D_{\theta_\tau\om}^{2\beta}([0, T], \RR^d)$.

Furthermore, the composition of a smooth function with a controlled rough path is yet a controlled rough path.

\begin{lemma}\label{l7.3}
Assume that $(Y, Y')\in \mathscr D_\om^{2\beta}([0,T], \RR^d)$ and $f\in C^2_b (\RR^d, \RR^{d\times m})$. Then $f(Y)$ is also controlled by $\omega$ with
\begin{align*}
(f(Y(t)))'&=f'(Y(t))Y'(t), \\
R^{f(Y)}(s,t)&=f(Y(t))-f(Y(s))-f'(Y(s))Y'(s)(\omega(t)-\omega(s)).
\end{align*}
\end{lemma}
For the proof we refer to \cite{MR3289027}, Lemma 7.3.\\

In what follows we would like to define the integral of a controlled rough path. The idea is to define it by using compensations of Riemann sums, which makes sense due to the following result, known as sewing lemma.

Define the space $\mathcal C^{\beta, \gamma}(\Delta[0, T], \RR^m)$
of functions $\Xi$ from $\Delta[0, T]$ into $\RR^m$ with $\Xi(t, t)=0$ and
\begin{equation*}
\|\Xi\|_{\beta, \gamma}=\ltn \Xi\rtn_\beta+\ltn \delta \Xi\rtn_\gamma<\infty,
\end{equation*}
where $\delta\Xi(s,u,t):=\Xi(s,t)-\Xi(s,u)-\Xi(u,t)$ with
$\ltn \delta\Xi\rtn _\gamma:=\sup\limits_{s<u<t}\frac{|\delta \Xi(s,u,t)|}{|t-s|^\gamma}$.
\begin{lemma} (Sewing lemma, \cite[Lemma 4.2]{MR3289027}).
Assume $\beta \in (1/3,1/2)$ and $\gamma>1$. Then there exists a unique continuous map
$\mathcal I: \mathcal C^{\beta, \gamma}(\Delta[0, T], \RR^m)\to \mathcal C^{\beta}([0, T], \RR^m)$ such that $(\mathcal I\Xi)(0)=0$ and
\begin{equation}\label{I1}
(\mathcal I\Xi)(t)-(\mathcal I\Xi)(s)=\lim_{|\mathcal P(s,t)|\to 0}\sum_{[u,v]\in \mathcal P(s,t)}\Xi(u,v).
\end{equation}
\end{lemma}

As a consequence, it is easy to derive that under the assumptions of Lemma \ref{l7.3}, considering
$$\Xi(u,v)=f(Y(u))(\w(v)-\w(u))+f'(Y(u))Y'(u) \bbomega(u,v)$$
we can define the integral as follows
\begin{equation}\label{I2}
\int_s^t f(Y(r))d\bm\w(r)=\lim_{|\mathcal P(s,t)|\to 0}\sum_{[u,v]\in \mathcal P(s,t)}\Xi(u,v).
\end{equation}

Note that
\begin{align*}
&\sum_{[u,v]\in \mathcal P(s+\tau,t+\tau)}\big(f(Y(u))(\w(v)-\w(u))+f'(Y(u))Y'(u) \bbomega(u,v)\big)\\
&=\sum_{[u,v]\in \mathcal P(s,t)}\big(f(Y(u+\tau))(\w(v+\tau)-\w(u+\tau))+f'(Y(u+\tau))Y'(u+\tau)\bbomega(u+\tau,v+\tau)\big)
\\
&=\sum_{[u,v]\in \mathcal P(s,t)}\big(f(Y(u+\tau))(\theta_\tau\w(v)-\theta_\tau \w(u))+f'(Y(u+\tau))Y'(u+\tau)\theta_\tau\bbomega(u,v)\big)\\
&=\sum_{[u,v]\in \mathcal P(s,t)}\big(f(\tilde Y(u))(\theta_\tau\w(v)-\theta_\tau \w(u))+f'(\tilde Y(u))\tilde Y'(u)\theta_\tau\bbomega(u,v)\big),
\end{align*}
where in the last equality $\tilde Y$ denotes the function defined as $\tilde Y(u):=Y(u+\tau)$, for $u\in [s,t]$. Now, taking the limit as $|\mathcal P(s,t)|\to 0$ in the above expression (and thus also $|\mathcal P(s+\tau,t+\tau)|\to 0$), by \eqref{I2} and the definition of $\theta_\tau \bm \om=(\theta_\tau \omega,\theta_\tau\bbomega)$, we have
 \begin{align}\label{I3}
\int_{s+\tau}^{t+\tau}f(Y(r))d\bm \w(r)=\int_s^t f(Y(r+\tau))d\theta_\tau\bm \w(r)=\int_s^t f(\tilde Y(r))d\theta_\tau\bm \w(r).
\end{align}
The expression (\ref{I3}) represents the behavior of the integral under a change of varia\-ble, and this will be the key property, together with the pathwise character of the integral, to further establish the cocycle property for the RDS generated by rough differential equations, see Lemma \ref{l4} below.\\

Now we would like to solve rough differential equations driven by the rough path $\bm \om$.  For $T>0$, consider the equation
 \begin{align}\label{equ1}
          \left\{\begin{array}{rll}
          dY(t)&=&f(Y(t))d\bm \om(t), \quad t\in [0, T]\\
          Y(0)&=&\xi\in \RR^d.
          \end{array}          \right.
        \end{align}
The integral against $\bm \om$ has to be understood as in (\ref{I2}) above. The following result regarding the existence and uniqueness of a solution to (\ref{equ1}) can be found in \cite{MR3289027}, Theorem 8.4.
\begin{theorem}\label{t3}
Let us consider any $T>0$, $\beta \in (1/3,1/2)$, $f\in C_b^3(\RR^d, \RR^{d\times m})$
and $\xi\in \RR^d$. Then there is a unique solution
$(Y, Y')\in \mathscr D_\om^{2\beta}([0, T], \RR^d)$ to \eqref{equ1}, that is,
\begin{equation}\label{e1a1}
Y(t)=\xi+\int_0^t f(Y(s))d\bm \om(s),  \quad\forall t\in [0, T]
\end{equation}
with $Y'=f(Y)$.
\end{theorem}

In the next result we establish that the solution operator given by (\ref{e1a1}) is a cocycle.
\begin{lemma}\label{l4}
Over the metric dynamical system $(\Omega, \fF,\PP_{\frac 12},\theta)$ introduced in Example \ref{ex1}, the solution of equation \eqref{equ1} generates a random dynamical system $\varphi_0:\RR^+ \times \Omega \times \RR^m \mapsto \RR^m$ given by
$\varphi_0(t,\omega,\xi)=Y(t)$, for all $t\in [0, T]$ and $\omega\in \Omega$.
\end{lemma}

\begin{proof}
First of all, the measurability of $\varphi_0$ directly follows from the Picard iteration argument.
On the other hand, it is trivial to see that $\varphi_0(0,\omega,\xi)=\xi.$ Therefore, it remains to prove the cocycle property. Taking into account that the integral
defined by (\ref{I2}) is additive, and its behavior under a change of variable is given by (\ref{I3}), we have the following chain of equalities
\begin{align*}
\varphi_0(t+\tau,\omega,\xi)&=\xi+\int_0^{t+\tau} f(Y(s))d\bm \om(s)\\
&=\xi+\int_0^{\tau} f(Y(s))d\bm \om(s)+\int_\tau^{t+\tau} f(Y(s))d\bm \om(s)\\
&=Y(\tau)+\int_0^t f(Y(s+\tau))d\theta_\tau\bm \om(s).
\end{align*}
Considering again the auxiliary function $\tilde Y(t):=Y(t+\tau)$ for $t\geq 0$, the previous expression reads
\begin{align*}
\varphi_0(t+\tau,\omega,\xi)&=\tilde Y(0)+\int_0^t f(\tilde Y(s))d\theta_\tau\bm \om(s)\\
&=\varphi_0(t,\theta_\tau \omega,\varphi_0(\tau,\omega,\xi)),
\end{align*}
which completes the proof.
\end{proof}

\vskip.3cm
In what follows, we would like to compare the solution of \eqref{equ1} with that of the corresponding system driven by the approximated lift $ \bm\om_\de$, that is to say, we consider the following rough differential equation
 \begin{align}\label{equ2}
          \left\{\begin{array}{rll}
          dY_\de(t)&=&f(Y_\de(t))d\bm \om_\de(t), \quad t\in [0, T],\\
          Y_\de(0)&=&\xi_\de \in \RR^d,
          \end{array}          \right.
        \end{align}
with the aim of establishing the relationship with system (\ref{equ1}).

We should remark that $\om_\de$ is in $C^1([0, T], \RR^m)$ and therefore we can use the ordinary differential equations theory to interpret this equation. This actually means that (\ref{equ2}) is equivalent to the following
random ordinary differential equation
 \begin{align}\label{equ3}
          \left\{\begin{array}{rll}
          dY_\de(t)&=&f(Y_\de(t))d\om_\de(t), \quad t\in [0, T],\\
          Y_\de(0)&=&\xi_\de \in \RR^d,
          \end{array}          \right.
        \end{align}
Then under the regularity assumption $f\in C^3_b(\RR^d,  \RR^{d\times m})$ (which indeed is too much regularity for solving the ordinary differential equation (\ref{equ3})), there exists a unique $Y_\de \in C^{\beta}([0,T],\RR^d)$ solving \eqref{equ2}, that furthermore generates an RDS $\varphi_\de$ given by $\varphi_\de(t,\omega,\xi_\de)=Y_\de(t)$ for all $t\in [0, T]$ and $\omega\in \Omega$.\\

Now we are in position to establish the relationship between the two RDS $\varphi_\de$ and $\varphi_0$.
\begin{theorem}\label{m1}
Let $f\in C_b^3(\RR^d, \RR^{d\times m})$. For any given $T>0$ and $\beta\in (1/3,1/2)$,
there exists a sequence of positive numbers $(\delta_i)_{i\in\NN}$ (say dyadic numbers) converging to zero
such that if the sequence of initial conditions $(\xi_{\de_i})_{i\in\NN}$ converges to $\xi$, then the RDS $\varphi_{\de_i}$ generated by the solution of system (\ref{equ2}),
converges to $\varphi_0$, the RDS generated by the solution of (\ref{equ1}), in the space $C^\beta([0, T], \RR^d)$ when $i\to \infty$.
\end{theorem}

\begin{proof}
Since $\bm \w$ and $\bm \w_{\de_i}$ are rough paths, there exists a positive constant $K=K(\om)$ such that
$$\ltn \bm \w\rtn_{\beta,0,T} \leq K,\quad \ltn \bm \w_{\de_i}\rtn_{\beta,0,T} \leq K.$$
Now, by Theorem 8.5 and Remark 8.6 in \cite{MR3289027}, we know that there exists
$C=C(K, \beta, f)>0$ such that
\begin{align*}
\ltn Y_{\de_i}-Y\rtn_{\beta,0,T}\le C\Big(|\xi_{\de_i}-\xi|+\rho_{\beta,0, T} (\bm\w_{\de_i},\bm \w)\Big)
\end{align*}
hence Theorem \ref{l3} implies
$$\ltn \varphi_{\de_i}-\varphi_0\rtn_{\beta,0,T}\le C\Big(|\xi_{\de_i}-\xi|+\rho_{\beta,0, T} (\bm\w_{\de_i},\bm \w)\Big)\to 0,$$
when $i\to \infty$, which completes the proof.
\end{proof}


\section*{Appendix: Covariances of $\omega_\delta$ and $X_\delta$ for the fBm with any $H\in (0,1)$}\label{appn}

In this section, we would like to compute the explicit expression of the covariance of $\omega_\delta$ as well as that of $X_\delta$. In order to make the result as much general as possible, in this section we assume that $\omega$ is a fractional Brownian motion with Hurst parameter $H\in (0,1)$. In that way, in a forthcoming research, we will able to use the results when considering a fractional Brownian motion as driven noise.

\begin{theorem}\label{c1}
The covariance of  $X_\de$ is $\sigma_{X_\de}^2(u)=K(u){\rm id}$ with $K(u)$ given by
\begin{eqnarray*}
\begin{split}
K(u)
&=\frac{1}{H+1}\delta^{2H}+\frac 1{\delta^2(2H+1)(2H+2)}\bar K(u),
\end{split}
\end{eqnarray*}
where for $u\ge\de$,
\begin{eqnarray*}
\begin{split}
\bar K(u)&=(u+\de)^{2H+2}+(u-\de)^{2H+2}-2u^{2H+2}\\
&\qquad-\de(2H+2)((u+\de)^{2H+1}-(u-\de)^{2H+1})+\de^2(2H+1)(2H+2)u^{2H},
\end{split}
\end{eqnarray*}
and for $0\le u<\de$,
\begin{eqnarray*}
\begin{split}
\bar K(u)&=(u+\de)^{2H+2}+(\de-u)^{2H+2}-2u^{2H+2}\\
&\qquad-\de(2H+2)((u+\de)^{2H+1}+(\de-u)^{2H+1})+\de^2(2H+1)(2H+2)u^{2H}.
\end{split}
\end{eqnarray*}
\end{theorem}

The proof of this result is based on the following lemma.

\begin{lemma}\label{l2}
Assume that $\omega$ is a fractional Brownian motion with values in $\RR^m$ and Hurst parameter $H\in (0,1)$. Then for $u\ge 0$ the cox of $\omega_\de(u)$ is given by $I(u){\rm id}$, where
\begin{align*}
         I(u):= &\frac{1}{\delta^2(2H+1)(2H+2)}
     \\
         &\times\left\{\begin{array}{lcr}((u+\delta)^{2H+2}-2\delta^{2H+2}
         -2u^{2H+2}+(u-\de)^{2H+2}),& u\ge\delta&\\
          ((u+\delta)^{2H+2}-2\delta^{2H+2}
         -2u^{2H+2}+(\delta-u)^{2H+2}),&u<\delta&
          \end{array}          \right.
 \end{align*}
Furthermore, for $u\ge 0$, $\EE\omega_\de(u) \omega(u)$ is given by $J(u){\rm id}$, being
          \begin{align*}
          J(u):=&\frac{1}{{2}\delta(2H+1)}\times\left\{\begin{array}{lcr}(u+\delta)^{2H+1}-2\delta^{2H+1}-(u-\delta)^{2H+1}, &u\ge\delta&\\
          (u+\delta)^{2H+1}-2\delta^{2H+1}+(\delta-u)^{2H+1}, &u<\delta.&
          \end{array}          \right.
        \end{align*}
 \end{lemma}

\begin{proof}
We start calculating the covariance of $\omega_\delta$.  First, note that for $1\le i, j\le m$,
\begin{equation*}
 \omega^i_\delta(u)\omega^j_\delta(u)=\frac{1}{\delta^2}\int_0^u\int_0^u\theta_r\omega^i(\delta)\theta_q\omega^j(\delta)dqdr.
\end{equation*}
For $i\neq j$, $\EE(\omega^i_\delta(u)\omega^j_\delta(u))=0$ due to the independence of the components of $\omega$. If $i=j$,
\begin{align*}
  I(u):=\EE(\omega^i_\delta(u))^2&=\frac{1}{2\delta^2}\int_0^u\int_0^u|r-\delta-q|^{2H}  +|r+\delta-q|^{2H} -2|r-q|^{2H}dqdr\\
  &=:I_1(u)+I_2(u)+I_3(u).
\end{align*}
Notice that we have
\begin{align*}
 I_3(u)& = -\frac{2}{\delta^2}\int_0^u\int_0^r(r-q)^{2H}dqdr
 =-\frac{2}{\delta^2(2H+1)}\int_0^ur^{2H+1}dr\\
 &=-\frac{2}{\delta^2(2H+2)(2H+1)}u^{2H+2}.
\end{align*}

\medskip

Now we study $I_1$. First of all, consider the case $\delta<u$ and suppose $r-q>\delta$.
Then we have
\begin{align*}
  \int_{q+\delta}^u(r-q-\delta)^{2H}dr=\frac{1}{2H+1}(u-q-\delta)^{2H+1},
\end{align*}
that it is well-defined since $\delta+q<u$.  Considering now $r-q<\delta$ we obtain
\begin{equation*}
  \int_0^{q+\delta}(q+\delta-r)^{2 H}dr=\frac{1}{2H+1}(q+\delta)^{2H+1}
\end{equation*}
for $\delta+q<u$ and
\begin{equation*}
  \int_0^u(q+\delta-r)^{2 H}dr
  =\frac{1}{2H+1}\bigg((q+\delta)^{2H+1}-(q+\delta-u)^{2H+1}\bigg)
\end{equation*}
for $\delta+q>u$. Collecting all the previous cases, when $\delta <u$ we obtain that $I_1$ equals to
\begin{align*}
  I_1(u)= & \frac{1}{2\delta^2(2H+1)}\int_0^{u-\delta}((u-q-\delta)^{2H+1}+(q+\delta)^{2H+1})dq\\
 &+\frac{1}{2\delta^2(2H+1)}\int_{u-\delta}^u((q+\delta)^{2H+1}-(q+\delta-u)^{2H+1})dq\\
  =&\frac{1}{2\delta^2(2H+1)(2H+2)}((u+\delta)^{2H+2}-2\delta^{2H+2}+(u-\delta)^{2H+2}).
\end{align*}
On the other hand, when $\delta>u$,
\begin{align*}
 I_1(u)=&\frac{1}{2\delta^2}\int_0^u\int_0^u(q+\delta-r)^{2H}drdq\\
 =&\frac{1}{2\delta^2}\int_0^u\frac{1}{2H+1}((q+\delta)^{2H+1}-(q+\delta-u)^{2H+1})dq\\
 =&\frac{1}{2\delta^2(2H+1)(2H+2)}((u+\delta)^{2H+2}-2\delta^{2H+2}+(\delta-u)^{2H+2}),
\end{align*}
and therefore for $\delta>u$ we conclude that
\begin{align*}
 I_1(u)=&\frac{1}{2\delta^2(2H+1)(2H+2)}((u+\delta)^{2H+2}-2\delta^{2H+2}+(\delta-u)^{2H+2}).
\end{align*}

Finally we deal with $I_2(u)$. As before, assume in a first step that $\delta<u$. In addition, suppose that $r-q>-\delta$. Then for $q<\delta$, we have
\begin{equation*}
  \int_0^u(r-q+\delta)^{2H}dr=\frac{1}{2H+1}\bigg((u-q+\delta)^{2H+1}-(\delta-q)^{2H+1}\bigg)
\end{equation*}
while for $q>\delta$,
\begin{equation*}
  \int_{q-\delta}^u(r-q+\delta)^{2H}dr=\frac{1}{2H+1}(u-q+\delta)^{2H+1}.
\end{equation*}
If we suppose that $r-q<-\delta$, then we obtain
\begin{equation*}
  \int_0^{q-\delta}(q-\delta-r)^{2H}dr=\frac{1}{2H+1}(q-\delta)^{2H+1}
\end{equation*}
that it is well-defined since in that case $q>\delta$.
Collecting all the previous cases, when $\de<u$ we obtain that
\begin{align*}
  I_2(u) = &\frac{1}{2\delta^2(2H+1)}\int_0^\delta ( (u-q+\delta)^{2H+1}-(\delta-q)^{2H+1})dq\\
 & + \frac{1}{2\delta^2(2H+1)}\int_\delta^u((u-q+\delta)^{2H+1}+(q-\delta)^{2H+1}) dq\\
  =&\frac{1}{2\delta^2(2H+1)(2H+2)}((u+\delta)^{2H+2}-2\delta^{2H+2}+(u-\de)^{2H+2}).
\end{align*}
On the other hand, when $\delta>u$,
\begin{align*}
  I_2(u) &=\frac{1}{2\delta^2}\int_0^u\int_0^u(r+\delta-q)^{2H}drdq\\
  &=\frac{1}{2\delta^2(2H+1)}\int_0^u((u+\delta-q)^{2H+1}-(\delta-q)^{2H+1})dq\\
  &=\frac{1}{2\delta^2(2H+1)(2H+2)}((u+\delta)^{2H+2}-2\delta^{2H+2}+(\delta-u)^{2H+2}).
\end{align*}

As a result, it turns out that the covariance of $\omega_\de$ is the $m\times m$ matrix with diagonal elements
$I(u)$ and off-diagonal elements equal to zero.

\vskip.5cm

Now we consider $\EE  (\omega_\delta(u)\omega(u))$. As before, when $i\neq j$, $\EE  (\omega^i_\delta(u)\omega^j(u))=0$, hence we only consider the case $i=j$. We have that
\begin{align*}
 J(u):= \EE  (\omega^i_\delta(u)\omega^i(u))&=\EE\Big(\frac{1}{\delta}\int_0^u\theta_r\omega^i(\delta)\omega^i(u)dr\Big)\\
  &=\frac{1}{2\delta}\int_0^u\big((r+\delta)^{2H}-|r+\delta-u|^{2H}-|r|^{2H}+|r-u|^{2H}\big)dr\\
  &=\frac{1}{2\delta(2H+1)}((u+\delta)^{2H+1}-\delta^{2H+1})-\frac{1}{2\delta}J_1(u),
\end{align*}
where, for $\delta<u$,
\begin{align*}
  J_1(u)=&\int_{u-\delta}^u(r+\delta-u)^{2H}dr+\int_0^{u-\delta}(u-\delta-r)^{2H}dr\\
  =&\frac{1}{2H+1}(\delta^{2H+1}+(u-\delta)^{2H+1})
\end{align*}
and for $\delta>u$,
\begin{equation*}
  J_1(u)=\frac{1}{2H+1}(\delta^{2H+1}-(\delta-u)^{2H+1}).
\end{equation*}

As a consequence,
     \begin{align*}
          J(u)=&\frac{1}{{2}\delta(2H+1)}\times\left\{\begin{array}{lcr}(u+\delta)^{2H+1}-2\delta^{2H+1}-(u-\delta)^{2H+1}, &u>\delta&\\
          (u+\delta)^{2H+1}-2\delta^{2H+1}+(\delta-u)^{2H+1}, &u<\delta.&
          \end{array}          \right.
        \end{align*}

Therefore, $\EE  (\omega_\delta(u)\omega(u))$
is an $m\times m$ matrix with diagonal elements
{$J(u)$} and off-diagonal elements equal to zero.

\end{proof}
Now we can prove the main result of this section.

\begin{proof} (of Theorem \ref{c1}). Let us consider $u\geq 0$. Since $X_\delta=\omega_\delta-\omega$, then we know that for any $1\leq i\leq j\leq m$,
$$\EE(X^i_\delta (u)X^j_\delta(u))=\EE(\omega^i_\delta(u) \omega_\delta^j(u))+\EE(\omega^i (u) \omega^j(u))-\EE(\omega^i_\delta(u) \omega^j(u))-\EE(\omega^j_\delta(u) \omega^i(u)).$$
Then the result follows simply taking into account Lemma \ref{l2} and that $\sigma^2_\omega(u)=|u|^{2H}{\rm id}$.
\end{proof}

\begin{lemma}\label{l5}
When $u<0$ the formulas in Theorem \ref{c1} and Lemma \ref{l2} hold repla\-cing $u$ by $-u$.
\end{lemma}

\begin{proof}
Let $u<0$. Then it is easy to check that
\begin{align*}
  &\EE\bigg(\int_0^u\theta_r\omega(\delta)dr\int_0^u\theta_r\omega(\delta)dr\bigg)=\EE\bigg(\int_{-u}^0\theta_{r+u}\omega(\delta)dr\int_{-u}^0\theta_{r+u}\omega(\delta)dr\bigg)  \\
  =&\EE\bigg(\int^0_{-u}\theta_{r}\omega(\delta)dr\int_{-u}^0\theta_{r}\omega(\delta)dr\bigg) = \EE\bigg(\int_0^{-u}\theta_{r}\omega(\delta)dr\int_0^{-u}\theta_{r}\omega(\delta)dr\bigg).
\end{align*}
Similarly
\begin{align*}
    &\EE\bigg(\omega(u) \int_0^u\theta_r\omega(\delta)dr\bigg)= \EE\bigg((-\theta_u\omega(-u))\int_{-u}^0\theta_{r+u}\omega(\delta)dr\bigg)\\
=& - \EE\bigg(\omega(-u) \int_{-u}^0\theta_r\omega(\delta)dr\bigg)=\EE\bigg(\omega(-u)\int_0^{-u}\theta_r\omega(\delta)dr\bigg).
\end{align*}
 \end{proof}

In the particular case of dealing with the Brownian motion, that is, when $H=\frac{1}{2}$, then Theorem \ref{c1} reads as follows.
\begin{corollary}\label{c2} If $\omega$ denotes the Brownian motion, then the corresponding cova\-riance $\sigma_{X^i_\de}^2$ of each component $X^i_\de$ is given by
\begin{equation*}
\sigma_{X^i_\de}^2(u)=\left\{
\begin{array}{lcr}
u-\frac{u^3}{3\delta^2}, &\quad& 0\le u<\delta,\\
\frac 2 3\de,&\quad& u\ge \delta,
\end{array}
\right.\end{equation*}
for $1\le i\le m$.
\end{corollary}


\section*{Acknowledgements}
M.J. Garrido-Atienza and B. Schmalfu\ss \, would like to thank  VIASM in Hanoi (Vietnam) giving them the possibility for many fruitful discussions about the objectives of this article. The authors also would like to thank Robert Hesse for a critical proofreading of the paper.

H. Gao was supported by NSFC Grant No. 11531006.

M.J. Garrido-Atienza was supported by grants PGC2018-096540-I00 and US-1254251.

A. Gu likes to thank for a DAAD-K.C. Wong grant allowing him to work one semester at the FSU Jena (Germany).

\bibliographystyle{plain}
\bibliography{wong_zakai-rds-240220}
\end{document}